\documentclass[12pt]{amsart}
\usepackage[foot]{amsaddr} 
\usepackage[utf8]{inputenc}
\usepackage{amsmath, amssymb, amsthm}
\usepackage{geometry}
\usepackage{hyperref}
\usepackage{graphicx}
\usepackage{float}
\usepackage{seqsplit}
\usepackage{tikz}
\usepackage[dvipsnames]{xcolor}
\usepackage{marvosym}
\usepackage{lineno}

\newtheorem{theorem}{Theorem}[section]
\newtheorem{lemma}[theorem]{Lemma}
\newtheorem{proposition}[theorem]{Proposition}

\theoremstyle{definition}
\newtheorem{definition}[theorem]{Definition}
\newtheorem{example}[theorem]{Example}

\newtheorem{conjecture}[theorem]{Conjecture}

\usepackage{booktabs}
\usepackage{tikz}
\usepackage{float}
\usepackage{xcolor}

\usepackage{algorithm}
\usepackage{algpseudocode}
\makeatletter
\newenvironment{breakablealgorithm}
{
		\begin{center}
			\refstepcounter{algorithm}
			\hrule height.8pt depth0pt \kern2pt
			\parskip 0pt
			\renewcommand{\caption}[2][\relax]{
				{\raggedright\textbf{\fname@algorithm~\thealgorithm} ##2\par}%
				\ifx\relax##1\relax 
				\addcontentsline{loa}{algorithm}{\protect\numberline{\thealgorithm}##2}%
				\else 
				\addcontentsline{loa}{algorithm}{\protect\numberline{\thealgorithm}##1}%
				\fi
				\kern2pt\hrule\kern2pt
			}
		}
		{
		\kern2pt\hrule\relax
	\end{center}
}
\makeatother

\DeclareMathOperator{\MNM}{MNM}

\begin{document}

\title{An Algebraic Matching Algorithm for the Genus of Grid Diagrams}

\author{Jun Wang$^1$$^\textsuperscript{\Letter}$ }
\author{Shasha Yang$^1$}
\address{$^1$School of Mathematical Sciences, Hebei Normal University}
\address{$^\textsuperscript{\Letter}$Corresponding author}


\maketitle

\begin{abstract}
Grid diagrams encode knots and links combinatorially through an $n\times n$ winding matrix $W(G)$. A fundamental invariant of the grid diagram is its genus $g(G)$, which records essential topological information about the underlying knot. We give an algebraic algorithm that computes $g(G)$ directly from $W(G)$. The method repurposes the Hungarian algorithm (Kuhn, 1955; Munkres, 1957): after a column normalization that makes the matrix non-negative, it applies row–column reductions and minimum-vertex-cover corrections until the zero entries admit a perfect matching, equivalently until the minimum marking number $\mathrm{MNM}$ equals $n$. By Kőnig's theorem the resulting permutation defines a distinguished grid state, and $g(G)=\frac12\theta(H)-\frac{n-1}{2}$ is then read from this matrix via the classical grid-homology formula (Ozsváth–Stipsicz–Szabó). The novelty of this paper is the algorithm and its strongly polynomial $\mathcal{O}(n^{3})$ analysis; the underlying invariant is known grid-homology theory. We illustrate the method on the trefoil and additional examples, and discuss extensions to $m\times n$ matrices and grid-move navigation toward the knot genus $g(K)$.
\end{abstract}


\section{Introduction}

Grid diagrams provide a purely combinatorial encoding of knots and links: an
$n\times n$ grid diagram $G$ records the positions of $X$ and $O$ markers,
and the horizontal and vertical segments joining them form a knot diagram
(Cromwell~\cite{cromwell1995,cromwell2004}; Dynnikov~\cite{dynnikov2006}). 
In 2009, C.Manolescu, P.Ozsv\'ath, and S.Sarkar \cite{Manolescu2009} found a purely combinatorial description of knot floer homology by grid diagram. In 2025, C.Manolescu, P.Ozsv\'ath and D.Thurston \cite{Manolescu2025} gave combinatorial descriptions of the Heegaard Floer homology groups for arbitrary three-manifolds via the grid diagram of link.

What we consider here is the genus of grid diagram.
Each grid diagram
carries a \emph{winding matrix} $W(G)$, an $n\times n$ integer matrix (first
column and last row identically zero) that records the relative winding of the
horizontal and vertical arcs. 
A central integer of $G$ is its associated genus $g(G)$
(the genus of a Seifert surface constructed from $G$). In the grid-homology framework of
Ozsv\'ath--Stipsicz--Szab\'o~\cite{ozsvath2012}, $g(G)$ is computed from a
\emph{minimal non-negative matrix} $H$ derived from $W(G)$ by row and column
operations.
The associated genus of a grid diagram realises the Seifert genus of the underlying knot: for every knot $K$ there exists a grid diagram $G$ with $g(G)=g(K)$ (see e.g.\ Ozsv\'ath--Stipsicz--Szab\'o~\cite{ozsvath2012}); computing it directly from $W(G)$ is the central problem addressed in this paper.

The contribution of this paper is algorithmic. We show that the relevant
matrix $H$ can be produced from $W(G)$ by a repurposed Hungarian algorithm,
and that the data required are computed in
strongly polynomial $\mathcal{O}(n^{3})$ time. We do \emph{not} re-prove that
$g(G)$ is determined by $H$; our result is the \emph{computation} of $H$ and of $g(G)$.

\begin{theorem}[Main algorithmic result]\label{thm:main}
	Let $G$ be a grid diagram of index $n$ with winding matrix $W(G)$. After the
	column normalization~\eqref{eq:colnorm} and Algorithm~\ref{alg:main}, the output
	is a matrix $H$ with $\MNM(H)=n$ that is a minimal non-negative matrix in the
	sense of Ozsv\'ath--Stipsicz--Szab\'o~\cite{ozsvath2012}. The genus $g(G)$ is
	then obtained from $H$ by the closed formula of Proposition~3.4.9
	in~\cite{ozsvath2012}, whose well-definedness is guaranteed by
	Corollary~3.4.10 in~\cite{ozsvath2012}. The algorithm performs this computation
	in strongly polynomial $\mathcal{O}(n^{3})$ time.
\end{theorem}

	Termination of Algorithm~\ref{alg:main} and the $\mathcal{O}(n^{3})$ bound are
	established in Section~\ref{sec:analysis}. That the output $H$ is a minimal
	non-negative matrix, and that $g(G)$ is read from $H$, which are known results in grid homology; we recall them and
	verify that our normalization~\eqref{eq:colnorm} places us in their setting in
	Section~\ref{sec:grid}.

As an application of the algorithm, the effect of grid moves on the associated genus of grid diagram was considered in Section \ref{sec:gridmoves}.  It is known that given a grid diagram $G$ representing knot $K$, after a sequence of grid moves, we will obtain a new grid diagram $G'$ satisfying $g(G')=g(K)$.  Therefore, the study of the associated genus of grid diagram, and in particular the effect of grid moves on the associated genus, is important for the study of knot genus. In \cite{lackenby2015}, Marc Lackenby proved that any diagram of the unknot with $c$ crossings may be reduced to the trivial diagram using at most $(236c)^{11}$ Reidemeister moves. Thus we have the following conjecture.

\begin{conjecture}
Given a grid diagram $G$ representing knot $K$. How can we find a finite sequence of grid moves to obtain a grid diagram $G'$ from $G$, which satisfying $g(G')=g(K)$? Or is there an upper bound on the grid moves?
\end{conjecture}

The algebraic procedure developed here is not new in the optimisation sense:
it is the Hungarian algorithm repurposed, looking for any permutation of zeros
rather than a minimum-cost one. What is new is the \emph{application to grid
	diagrams}: the winding matrix of a grid diagram encodes its genus $g(G)$, and
Algorithm~\ref{alg:main} computes $g(G)$ algebraically, by producing the
minimal non-negative matrix from which $g(G)$ is read via the grid-homology
formula of Ozsv\'ath--Stipsicz--Szab\'o (Proposition~3.4.9;
Corollary~3.4.10). 

The method extends without difficulty to $m\times n$ matrices, with the
termination condition $\MNM(A) = \min\{m,n\}$. Directions for further work
include adapting it to coloured grid graphs and to diagrams that encode knots, and relating the extracted genus more explicitly to the
Seifert genus $g(K)$ over all grid diagrams of a knot
(see e.g.\ Cromwell~\cite{cromwell1995} and Jin, Stoimenow \& Lee~\cite{jin2021}).  The minimal grid number $c(K)$ and the range of associated
genera $g(G)$ over grid diagrams of $K$ have been studied extensively, e.g.\ by
Cromwell~\cite{cromwell1995} and Jin, Stoimenow \& Lee~\cite{jin2021}.

The paper is organised as follows.
Section~\ref{sec:alg} develops the algebraic algorithm (minimum marking
number, the iterative procedure, and why a maximum matching survives each
correction). Section~\ref{sec:analysis} establishes termination via a
potential-function argument and a strongly polynomial $\mathcal{O}(n^{3})$ time
bound. Section~\ref{sec:grid} applies the algorithm to grid diagrams, recalls
the grid-homology formula that yields $g(G)$, and illustrates it on the trefoil
and further examples. Section~\ref{sec:gridmoves} applied the algorithm to find the relation between associated genus of grid diagram and gird move. This is important for  computing the genus of knot.  

\section{Algorithm}\label{sec:alg}

Let $A$ be an $n\times n$ matrix.  Write $Z(A)=\{(i,j):a_{ij}=0\}$ for its
set of zero positions.

\begin{definition}\label{def:MNM}
	A pair $(R,C)$ with $R,C\subseteq\{1,\dots,n\}$ \emph{marks all zeros} of
	$A$ if for every $(i,j)\in Z(A)$, either $i\in R$ or $j\in C$ (or both).
	The \emph{minimum marking number} of $A$ is defined by 
	\[
	\MNM(A):=\min\{|R|+|C|\mid (R,C)\text{ marks all zeros of }A\}.
	\]
\end{definition}

By K\H{o}nig's theorem, $\MNM(A)$ equals the size of a maximum matching in
the bipartite graph with bipartition $\{1,\dots,n\}\sqcup\{1,\dots,n\}$ and
edge set $Z(A)$. One can therefore compute $\MNM(A)$ in $\mathcal{O}(n^{5/2})$
time via the Hopcroft--Karp algorithm~\cite{hopcroft1973}, or by any other
maximum matching algorithm.

We now describe the algorithm.

The algorithm has two phases.  Phase~1 ensures that every row and every column of  the matrix contains at least one zero.
Phase~2 iteratively adjusts the $n\times n$ matrix until $\MNM=n$.

\vspace{1cm}
\begin{breakablealgorithm}
	\caption{Transforming a non-negative matrix to one with $\MNM(H)=n$}
	\label{alg:main}
	\begin{algorithmic}[1]
		\Require An $n\times n$ non-negative integer matrix $W^{0}$
		\Ensure A non-negative matrix $H$ with $\MNM(H)=n$
		
		\Statex
		\State \textbf{Phase~1: Initial reduction.}
		\For{$i=1$ \textbf{to} $n$}
		\State $u_i\gets \min_{1\le j\le n} w^{0}_{ij}$
		\For{$j=1$ \textbf{to} $n$}
		\State $w^{1}_{ij}\gets w^{0}_{ij} - u_i$
		\EndFor
		\EndFor
		\For{$j=1$ \textbf{to} $n$}
		\State $v_j\gets \min_{1\le i\le n} w^{1}_{ij}$
		\For{$i=1$ \textbf{to} $n$}
		\State $w^{2}_{ij}\gets w^{1}_{ij} - v_j$
		\EndFor
		\EndFor
		
		\Statex
		\State \textbf{Phase~2: Incremental Hungarian (K\H{o}nig) iteration.}
		\State Initialise dual potentials $u_i,v_j$ from Phase~1 (the current reduced
		\Statex \quad matrix is $w^{2}_{ij}-u_i-v_j$, non-negative with a zero in each row/column).
		\State Compute an initial maximum matching $M$ in the zero graph
		\Statex \quad $E=\{(i,j):w^{2}_{ij}-u_i-v_j=0\}$.
		\State $t\gets 2$
		\Loop
		\If{$\MNM(W^2)=n$}
		\State $H\gets (w^{2}_{ij}-u_i-v_j)$ \textbf{and terminate}
		\EndIf
		\State Grow an alternating tree $T$ from the unmatched rows of $M$ in $E$
		\If{an augmenting path is found}
		\State Augment $M$ along it; restart $T$ from the new unmatched rows
		\Comment{$\MNM$ increases by $1$}
		\Else
		\State $T$ defines a minimum vertex cover $(R_t,C_t)$;
		\Statex \quad set $I'_t\gets\{1,\dots,n\}\setminus R_t$, $J'_t\gets\{1,\dots,n\}\setminus C_t$
		\State $\varepsilon\gets\min\{w^{2}_{ij}-u_i-v_j\mid i\in I'_t,\;j\in J'_t\,\}$
		\Comment{$\varepsilon>0$, Lemma~\ref{lem:delta}}
		\State Update potentials: $u_i\gets u_i+\varepsilon$ for $i\in I'_t$;
		\Statex \quad $v_j\gets v_j-\varepsilon$ for $j\in C_t$
		\State The newly created zeros enter $E$ and are added to $T$; continue growing
		\Comment{no restart}
		\State $t\gets t+1$
		\EndIf
		\EndLoop
	\end{algorithmic}
\end{breakablealgorithm}

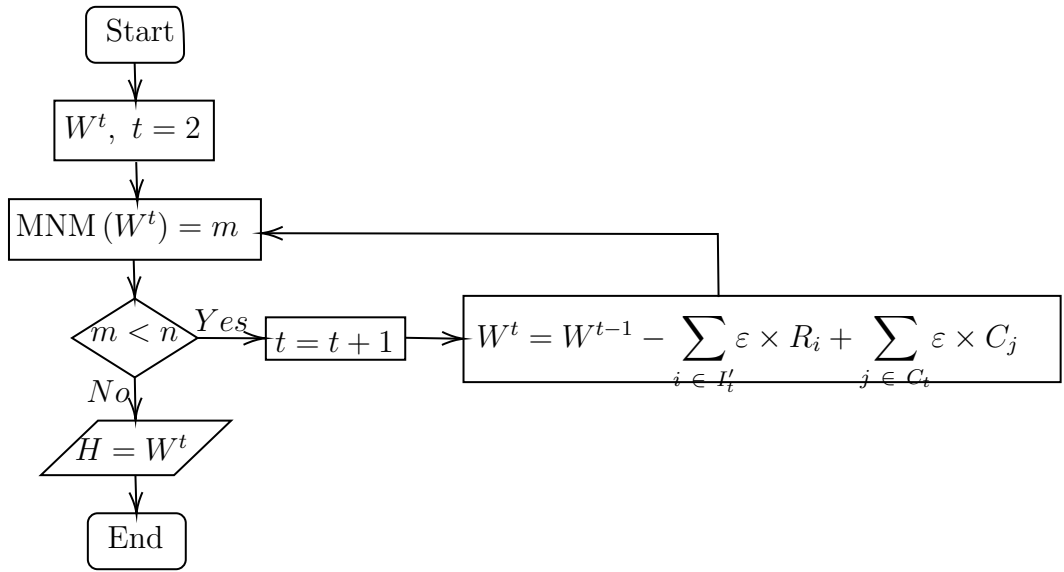
\begin{figure}[H]
	\centering
	\tikzset{every picture/.style={line width=0.75pt}}%
	\begin{tikzpicture}[x=0.75pt,y=0.75pt,yscale=-1,xscale=1]
		\draw    (80.44,80.14) -- (80.94,97.8) ;
		\draw [shift={(81,99.8)}, rotate = 268.36] [color={rgb, 255:red, 0; green, 0; blue, 0 }  ][line width=0.75]    (10.93,-3.29) .. controls (6.95,-1.4) and (3.31,-0.3) .. (0,0) .. controls (3.31,0.3) and (6.95,1.4) .. (10.93,3.29)   ;
		\draw   (80.52,198.49) -- (112.02,218.3) -- (80.52,238.1) -- (49.02,218.3) -- cycle ;
		\draw   (56.1,57.24) .. controls (56.1,54.53) and (58.3,52.33) .. (61.01,52.33) -- (100.33,52.33) .. controls (103.04,52.33) and (105.23,54.53) .. (105.23,75.47) .. controls (105.23,78.18) and (103.04,80.37) .. (100.33,80.37) -- (61.01,80.37) .. controls (58.3,80.37) and (56.1,78.18) .. (56.1,75.47) -- cycle ;
		\draw   (40.2,99.34) -- (120.2,99.34) -- (120.2,129.25) -- (40.2,129.25) -- cycle ;
		\draw   (17.4,148.72) -- (143.8,148.72) -- (143.8,179.1) -- (17.4,179.1) -- cycle ;
		\draw    (80.52,238.1) -- (80.87,257.25) ;
		\draw [shift={(80.9,259.25)}, rotate = 268.96] [color={rgb, 255:red, 0; green, 0; blue, 0 }  ][line width=0.75]    (10.93,-3.29) .. controls (6.95,-1.4) and (3.31,-0.3) .. (0,0) .. controls (3.31,0.3) and (6.95,1.4) .. (10.93,3.29)   ;
		\draw   (62.05,259.72) -- (128.86,259.72) -- (100.23,287.06) -- (33.41,287.06) -- cycle ;
		\draw   (56.97,311.07) .. controls (56.97,308.36) and (59.17,306.16) .. (61.88,306.16) -- (101.19,306.16) .. controls (103.91,306.16) and (106.1,308.36) .. (106.1,311.07) -- (106.1,329.29) .. controls (106.1,332) and (103.91,334.2) .. (101.19,334.2) -- (61.88,334.2) .. controls (59.17,334.2) and (56.97,332) .. (56.97,329.29) -- cycle ;
		\draw    (112.02,218.3) -- (144.03,218.46) ;
		\draw [shift={(146.03,218.47)}, rotate = 180.29] [color={rgb, 255:red, 0; green, 0; blue, 0 }  ][line width=0.75]    (10.93,-3.29) .. controls (6.95,-1.4) and (3.31,-0.3) .. (0,0) .. controls (3.31,0.3) and (6.95,1.4) .. (10.93,3.29)   ;
		\draw   (146.27,207.95) -- (216.2,207.95) -- (216.2,229.22) -- (146.27,229.22) -- cycle ;
		\draw    (216.2,218.2) -- (243.6,218.67) ;
		\draw [shift={(245.6,218.7)}, rotate = 180.98] [color={rgb, 255:red, 0; green, 0; blue, 0 }  ][line width=0.75]    (10.93,-3.29) .. controls (6.95,-1.4) and (3.31,-0.3) .. (0,0) .. controls (3.31,0.3) and (6.95,1.4) .. (10.93,3.29)   ;
		\draw   (245.44,197.2) -- (545,197.2) -- (545,240.43) -- (245.44,240.43) -- cycle ;
		\draw    (373.4,197.8) -- (373.02,166.18) -- (353,166.18) -- (145.8,165.8) ;
		\draw [shift={(143.8,165.8)}, rotate = 0.1] [color={rgb, 255:red, 0; green, 0; blue, 0 }  ][line width=0.75]    (10.93,-3.29) .. controls (6.95,-1.4) and (3.31,-0.3) .. (0,0) .. controls (3.31,0.3) and (6.95,1.4) .. (10.93,3.29)   ;
		\draw    (79.96,178.84) -- (80.46,196.5) ;
		\draw [shift={(80.52,198.49)}, rotate = 268.36] [color={rgb, 255:red, 0; green, 0; blue, 0 }  ][line width=0.75]    (10.93,-3.29) .. controls (6.95,-1.4) and (3.31,-0.3) .. (0,0) .. controls (3.31,0.3) and (6.95,1.4) .. (10.93,3.29)   ;
		\draw    (80.84,286.94) -- (81.34,304.6) ;
		\draw [shift={(81.4,306.6)}, rotate = 268.36] [color={rgb, 255:red, 0; green, 0; blue, 0 }  ][line width=0.75]    (10.93,-3.29) .. controls (6.95,-1.4) and (3.31,-0.3) .. (0,0) .. controls (3.31,0.3) and (6.95,1.4) .. (10.93,3.29)   ;
		\draw    (81.24,130.14) -- (81.74,147.8) ;
		\draw [shift={(81.8,149.8)}, rotate = 268.36] [color={rgb, 255:red, 0; green, 0; blue, 0 }  ][line width=0.75]    (10.93,-3.29) .. controls (6.95,-1.4) and (3.31,-0.3) .. (0,0) .. controls (3.31,0.3) and (6.95,1.4) .. (10.93,3.29)   ;
		
		\draw (64.44,57.04) node [anchor=north west][inner sep=0.75pt]   [align=left] {Start};
		\draw (43.64,104.26) node [anchor=north west][inner sep=0.75pt]    {$W^{t} ,\ t=2$};
		\draw (19.35,153.3) node [anchor=north west][inner sep=0.75pt]    {$\MNM\left( W^{t}\right) =m$};
		\draw (56.87,209.95) node [anchor=north west][inner sep=0.75pt]    {$m< n$};
		\draw (54.81,237.77) node [anchor=north west][inner sep=0.75pt]    {$No$};
		\draw (48.86,265.37) node [anchor=north west][inner sep=0.75pt]    {$H=W^{t}$};
		\draw (65.3,310.86) node [anchor=north west][inner sep=0.75pt]   [align=left] {End};
		\draw (149.47,213.75) node [anchor=north west][inner sep=0.75pt]    {$t=t+1$};
		\draw (250.52,204.33) node [anchor=north west][inner sep=0.75pt]    {$W^{t} =W^{t-1} -{\displaystyle \sum _{i\ \in \ I'_t}} \varepsilon \times R_{i} +{\displaystyle \sum _{j\ \in \ C_t}} \varepsilon \times C_{j}$};
		\draw (108.91,202.15) node [anchor=north west][inner sep=0.75pt]    {$Yes$};
	\end{tikzpicture}
	\caption{Flowchart of the iterative correction phase (Phase~2). The loop repeats the minimum-vertex-cover correction until the minimum marking number reaches $n$ (equivalently, the zero positions admit a perfect matching); $R_i$ ($C_j$,respectively)denotes the matrix with all elements equal to 0,except the $i$-th row ($j$-th column,respectively), where all elements are 1; the the returned non-negative matrix is $H=W^{t}$.}
	\label{fig:flowchart}
\end{figure}

For each row of  the Phase~1 (lines 2--13),  we subtract its minimum entry, then for each column we subtract its minimum entry. The resulting matrix $W^{2}$ is non-negative and each row and each column contains at least one zero.

At the start of Phase~2 (lines 15--32) we maintain, together with the current reduced matrix,
a maximal matching $M$ in its zero graph and the dual potentials $u_i,v_j$
introduced in Phase~1 (so that the current entry is $w^{2}_{ij}-u_i-v_j$).  If
$\MNM(W^2)=n$ the matching is perfect and we terminate.  Otherwise we grow an
alternating tree from an unmatched row.  If an augmenting path is found we
augment $M$ (raising $\MNM$ by one) and continue; if no augmenting path
exists, the tree yields a minimum vertex cover $(R_t,C_t)$ by K\H{o}nig's
theorem.  This is the standard alternating-path construction (see
e.g.~\cite[Ch.~7]{papadimitriou1982}) and has the following well-known
property:

\begin{lemma}\label{lem:delta}
	Let $(R_t,C_t)$ be the minimum vertex cover obtained from a maximum
	matching $M$ via alternating paths.  Let $I'_t$ and $J'_t$ be the rows and
	columns \emph{not} in the cover.  Then there is no zero of the current reduced
	matrix in $I'_t\times J'_t$.  Hence
	\[
	\varepsilon = \min\{w^{2}_{ij}-u_i-v_j\mid i\in I'_t,\;j\in J'_t\,\}\geq 1.
	\]
\end{lemma}

We then apply the dual update
\begin{itemize}
	\item decrease $u_i$ by $\varepsilon$ for every unmarked row $i\in I'_t$;
	\item increase $v_j$ by $\varepsilon$ for every marked column $j\in C_t$.
\end{itemize}
Equivalently, in the explicit matrix form this subtracts $\varepsilon$ from the
unmarked rows and adds $\varepsilon$ to the marked columns.  The operation is
the dual-update step of the Hungarian algorithm
(Kuhn~\cite{kuhn1955}; see also Burkard, Dell'Amico \& Martello~\cite{burkard2009}).
Crucially, the alternating tree, the matching $M$ and a slack array over the
tree frontier are \emph{maintained across corrections}: a correction only
exposes new zero edges on the frontier and the search continues without
restarting.  It is this incremental maintenance---not a from-scratch
recomputation of the matching at each step---that delivers the strongly
polynomial $\mathcal{O}(n^{3})$ bound of Section~\ref{sec:analysis}.

\begin{table}[H]
	\centering
	\caption{Effect of the correction step on matrix entries}
	\label{tab:update}
	\begin{tabular}{lccc}
		\toprule
		Row status & Column status & Update & Zeros preserved? \\
		\midrule
		$i\in I'_t$ (unmarked) & $j\in C_t$ (marked) & $-\varepsilon+\varepsilon=0$ & unchanged \\
		$i\in I'_t$ (unmarked) & $j\in J'_t$ (unmarked) & $-\varepsilon$ & new zeros appear \\
		$i\in R_t$ (marked)   & $j\in C_t$ (marked)   & $+\varepsilon$ & zeros disappear \\
		$i\in R_t$ (marked)   & $j\in J'_t$ (unmarked) & $0$     & unchanged \\
		\bottomrule
	\end{tabular}
\end{table}

After the update we increment $t$ and repeat.

\section{Termination and Complexity}\label{sec:analysis}

The termination proof requires that $\MNM$ never decreases along the
iteration.  The naive claim that the current maximum matching survives each
correction step is false in general; we therefore establish monotonicity from
the structure of \emph{minimal} markings instead.  Write $Z(A)$ for the set of
zero positions of a matrix $A$.

\subsection{Non-decrease of MNM}
\begin{lemma}\label{lem:hall-cover}
	Let $(R_A,C_A)$ be a minimal marking of all zeros of an $n\times n$ matrix
	$A$, where $R_A$ is the set of marked rows and $C_A$ the set of marked
	columns.  For any subset $B\subseteq R_A$, define its neighbour set in the
	unmarked columns by
	\[
	N(B)=\{j\mid j\notin C_A,\ \exists i\in B\text{ with }a_{i,j}\in Z(A)\}.
	\]
	Then $|N(B)|\ge |B|$.
\end{lemma}

\begin{proof}
	Suppose, for a contradiction, that some $B\subseteq R_A$ satisfies
	$|N(B)|<|B|$.  Build a new marking $(R_1,C_1)$ with
	$R_1=R_A\setminus B$ and $C_1=C_A\sqcup N(B)$.  We check that $(R_1,C_1)$
	marks every zero $z=a_{i,j}\in Z(A)$:
	\begin{itemize}
		\item if $i\in R_A\setminus B=R_1$, then $z$ is marked by the row $i\in R_1$;
		\item if $i\in B\subseteq R_A$: when $j\in C_A$ we have $j\in C_1$,
		and when $j\notin C_A$ the definition of $N(B)$ gives $j\in N(B)\subseteq C_1$;
		\item if $i\notin R_A$, minimality of the original marking forces $j\in C_A\subseteq C_1$.
	\end{itemize}
	Since $N(B)\cap C_A=\varnothing$, we have $|C_1|=|C_A|+|N(B)|$.  Let
	$m=|R_A|+|C_A|$ be $\MNM(A)$.  Then
	\[
	|R_1|+|C_1| = |R_A|-|B|+|C_A|+|N(B)| = m+\bigl(|N(B)|-|B|\bigr) < m,
	\]
	contradicting the minimality of $(R_A,C_A)$.  Hence $|N(B)|\ge|B|$.
\end{proof}

The symmetric statement also holds: for any $E\subseteq C_A$, writing
$N(E)=\{i\mid i\notin R_A,\ \exists j\in E\text{ with } a_{i,j}\in Z(A)\}$,
one has $|N(E)|\ge|E|$.

\begin{definition}\label{def:cross-zero}
	A zero of $A$ that is marked by \emph{both} a row in $R_A$ and a column in
	$C_A$ is called a \emph{cross zero}.
\end{definition}

\begin{proposition}\label{prop:invariant}
	Let $Z$ be the set of all zeros of $A$, and let $Z'$ be obtained from $Z$
	by deleting all cross zeros.  Then the minimum marking number of $Z'$ equals
	the minimum marking number of $Z$.
\end{proposition}

\begin{proof}
	Write $m=\MNM(Z)$ and $m_0=\MNM(Z')$.  Since $Z'\subseteq Z$, we have
	$m_0\le m$.  For the reverse inequality, take \emph{any} marking
	$(R',C')$ of $Z'$, and set $B=R_A\setminus R'$ and $E=C_A\setminus C'$.
	For a row $i\in B$ we have $i\notin R'$, so every zero
	$z=a_{i,j}\in Z'$ with $j\notin C_A$ must be column-marked, i.e.
	$j\in C'$, hence $j\in N(B)$; thus $N(B)\subseteq C'$.  Because the
	columns in $N(B)$ avoid $C_A$, they are disjoint from $C_A\cap C'$, so
	$|C'|\ge |C_A\cap C'|+|N(B)|$.  Symmetrically,
	$|R'|\ge |R_A\cap R'|+|N(E)|$.  Adding and using
	Lemma~\ref{lem:hall-cover} (and its symmetric form),
	\[
	|R'|+|C'|
	\ge |R_A\cap R'|+|E|+|C_A\cap C'|+|B|
	= |R_A|+|C_A| = m.
	\]
	Since $(R',C')$ was arbitrary, $m_0\ge m$, and therefore $m_0=m$.
\end{proof}

\begin{proposition}\label{prop:monotone}
	In each iteration of the algorithm the minimum marking number $\MNM$ stays
	the same or increases.
\end{proposition}

\begin{proof}
	Let $m$ be $\MNM$ before the correction, $m_0$ the $\MNM$ of the zeros
	remaining after deleting the cross zeros before the correction, and $m'$ the
	$\MNM$ after the correction.  If there are no cross zeros, the correction
	only creates new zeros and never removes old ones, so $m'\ge m$.  If cross
	zeros are present, Proposition~\ref{prop:invariant} gives $m_0=m$; the
	post-correction  $m'$ is the minimum marking number of the old non-cross zeros
	together with the newly created zeros, hence $m'\ge m_0=m$.  In either case
	$m'\ge m$, so $\MNM$ never decreases.
\end{proof}
\subsection{Termination}
For each matrix $W^t$, let $r_i^{(t)}$ be the total amount subtracted from row~$i$ of $W^2$ and $c_j^{(t)}$ the total amount added to column~$j$ of $W^2$. Then the matrix $W^{t}=(w_{i,j}^t)_{n\times n}$ can be expressed as
\begin{equation}\label{eq:Wt}
	w^{t}_{ij} = w^{2}_{ij} - r_i^{(t)} + c_j^{(t)},\qquad t\ge2.
\end{equation}

 Denote
	\[
	S_t := \sum_{i=1}^{n} r_i^{(t)} - \sum_{j=1}^{n} c_j^{(t)},
	\]
	with $S_2=0$.

\begin{lemma}\label{lem:potential-increase}
	If $\MNM(W^{t})=m<n$, then
	\[
	S_{t+1} = S_t + \varepsilon\,(n-m) \ge S_t + 1.
	\]
	Hence $S_t$ strictly increases when the algorithm has not yet terminated.
\end{lemma}

\begin{proof}
	The correction step subtracts $\varepsilon$ from the $|I'_t|$ unmarked rows and
	adds $\varepsilon$ to the $|C_t|$ marked columns. Therefore
	\[
	S_{t+1} - S_t = \varepsilon\,|I'_t| - \varepsilon\,|C_t|
	= \varepsilon\bigl(|I'_t| - |C_t|\bigr).
	\]
	Since $m = |R_t|+|C_t|$ and $|I'_t| = n - |R_t|$, we have
	$|I'_t| - |C_t| = n - m$. The assumption $m<n$ gives $n-m\ge1$, and
	$\varepsilon\geq 1$ by Lemma~\ref{lem:delta}, so the increment is at least~$1$.
\end{proof}

\begin{lemma}[Bounded potential]\label{lem:bounded}
	For every permutation $\sigma$ in the symmetric group $S_n$,
	\[
	S_t \le \sum_{i=1}^{n} w^{2}_{i,\sigma(i)}.
	\]
\end{lemma}

\begin{proof}
	From~\eqref{eq:Wt} and the non-negativity of $W^{t}$ we have
	$r_i^{(t)} - c_j^{(t)} \le w^{2}_{ij}$ for all $i,j$. Fix $\sigma$ and sum
	with $j=\sigma(i)$:
	\[
	\sum_{i=1}^{n} \bigl(r_i^{(t)} - c_{\sigma(i)}^{(t)}\bigr)
	\le \sum_{i=1}^{n} w^{2}_{i,\sigma(i)}.
	\]
	Because $\sigma$ is a bijection, $\sum_i c_{\sigma(i)}^{(t)} = \sum_j c_j^{(t)}$,
	and the left-hand side is $S_t$.
\end{proof}

\begin{proposition}\label{prop:increase}
	If $\MNM = m < n$ at some iteration, then $\MNM$ must strictly increase
	after finitely many additional iterations.
\end{proposition}

\begin{proof}
	Assume $\MNM$ stays constant at $m$ forever after that iteration. Then
	by Lemma~\ref{lem:potential-increase}, the potential $S_t$ increases by at
	least $1$ at each step and diverges to $+\infty$. But
	Lemma~\ref{lem:bounded} gives a finite upper bound (take any permutation,
	say the identity), a contradiction. Hence $\MNM$ cannot stay at $m$
	indefinitely.
\end{proof}

\begin{theorem}\label{thm:conv}
	For any $n\times n$ non-negative integer matrix $W^{0}$,
	Algorithm~\ref{alg:main} terminates with a matrix $H$ satisfying
	$\MNM(H)=n$.
\end{theorem}

\begin{proof}
	After Phase~1, $2\le\MNM(W^{2})\le n$.  Proposition~\ref{prop:monotone}
	shows that $\MNM$ never decreases. Proposition~\ref{prop:increase} shows
	that whenever $\MNM$ is below $n$, it must eventually rise. Since it can
	only take integer values bounded by $n$, it reaches $n$ after finitely many
	steps.
\end{proof}

\subsection{Time bound}

Let $M_0 = \max_{i,j} w^{2}_{ij}$ be the largest entry after Phase~1.
Lemmas~\ref{lem:potential-increase} and~\ref{lem:bounded} give a crude
pseudo-polynomial bound of $\mathcal{O}(n M_0)$ iterations between successive
increases of $\MNM$.  This is sharpened to a strongly polynomial bound by
implementing Phase~2 as the incremental Hungarian (K\H{o}nig) iteration
described above: the alternating tree, the matching $M$, and a slack array over
the tree frontier are maintained across corrections, so a single correction
that exposes new zeros costs only $\mathcal{O}(n)$ amortised work and adds at
least one new vertex to the tree.  Consequently at most $n$ corrections occur
between two augmentations, and there are at most $n$ augmentations in total.
Each augmenting-path search therefore terminates in $\mathcal{O}(n^{2})$ time,
and the overall worst-case running time is $\mathcal{O}(n^{3})$.  This is the
classical tight bound of the Hungarian method (Kuhn~\cite{kuhn1955}; see also
Burkard, Dell'Amico \& Martello~\cite{burkard2009}) and is independent of the
numerical magnitude $M_0$.

\section{Application 1: computing the genus of grid diagrams}\label{sec:grid}

A \emph{grid diagram} of size~$n$ places $n$ rows and $n$ columns of squares
in the plane, with one ``X'' marker and one ``O'' marker in each row and each
column. Connecting each X to the O in its column by a vertical segment and
each O to the X in its row by a horizontal segment, with crossing information
determined by the diagram, gives a knot diagram. We follow Cromwell~\cite{cromwell2004} in defining the \emph{winding matrix}
$W(G)$ of a grid diagram $G$. This is an $n\times n$ integer matrix whose
first column and last row are identically zero.  The matrix need not be
non-negative. To apply the algorithm we first add column constants:
\begin{equation}\label{eq:colnorm}
	d_j = -\min_{1\le i\le n} w_{ij}^G,\qquad
	w^{0}_{ij} = w_{ij}^G + d_j,\;\forall i.
\end{equation}
The algorithm then produces a matrix $H$ with $\MNM(H)=n$.  By
K\H{o}nig's theorem, $H$ admits a permutation $\sigma$ with
$h_{i,\sigma(i)} = 0$ for all $i$. And this matrix is minimal non-negative matrix by \cite[Lemma 3.4.2]{ozsvath2012}. This permutation encodes a distinguished
grid state.

The fact that $g(G)$ can be read from any minimal non-negative matrix $H$
derived from $W(G)$, and that its value is independent of the choice of $H$, is part of the grid-homology
framework of
Ozsv\'ath--Stipsicz--Szab\'o~\cite{ozsvath2012}. In that work, the Euler characteristic $\chi(F_H)$, and hence
the genus $g(G)$, as a function of the  $\theta(H)$ and the grid
number $n$;  then shows that $\chi(F_H)$, and therefore
$g(G)$, is independent of the particular minimal non-negative matrix chosen.
Our contribution is algorithmic: we show that the row--column reduction of
Algorithm~\ref{alg:main} produces exactly such a matrix $H$ from $W(G)$ and
computes the data needed for the formula  in
$\mathcal{O}(n^{3})$ time.

\begin{theorem}[Genus formula, Ozsv\'ath--Stipsicz--Szab\'o]\cite[Prop.~3.4.9]{ozsvath2012}\label{thm:genus}
	Let $H$ be a minimal non-negative matrix produced by
	Algorithm~\ref{alg:main} and let $n$ be the grid number of $G$.
	The Euler characteristic of the associated surface~$F_H$ is
	\begin{equation}\label{eq:chi}
		\chi(F_H)=n-\theta(H),
	\end{equation}
	where $\theta(H)$ is the invariant of $H$ defined in~\cite[Def.~3.4.8]{ozsvath2012}:
	for each square~$Q$ marked with an $X$ or an $O$ in the grid, let
	$\theta(Q,H)$ be the average of the four matrix entries
	$h_{i,j},\,h_{i,j+1},\,h_{i+1,j},\,h_{i+1,j+1}$ adjoining~$Q$
	(indices taken modulo~$n$ on the torus), and set
	\[
	\theta(H)=\sum_{X\in\mathbb{X}}\theta(X,H)+\sum_{O\in\mathbb{O}}\theta(O,H).
	\]
	Since $F_H$ is a connected orientable surface with one boundary component,
	$\chi(F_H)=1-2g(G)$; therefore the genus of the grid diagram is
	\begin{equation}\label{eq:genus}
		g(G)=\frac{1-\chi(F_H)}{2}=\frac{1}{2}\,\theta(H)-\frac{n-1}{2}.
	\end{equation}
\end{theorem}

%

\begin{example}[Trefoil knot]\label{ex:trefoil}
	The $5\times5$ grid diagram of the trefoil knot $3_1$ ($\sigma_{\mathbb{O}}=\{5,4,3,2,1\}, \sigma_{\mathbb{X}}=\{2,1,5,4,3\}, $) gives the winding
	matrix
	\[
	W(G) =
	\begin{pmatrix}
		0 & 1 &  1 &  0 &0\\
		0 &  1 & 2 &  1&0\\
		0 &  1 &  2 & 2 & 1\\
		0 &0 &1 &1 &1\\
		0 &  0 &  0 &  0 &0
	\end{pmatrix}.
	\]
\noindent\textbf{Phase~1:}	

All column minima and row minima  are $0$, so $W^{2}=W^{1}=W^0=W(G)$.
	
	\noindent\textbf{Phase~2:}

	Zero positions:
	$$
	(1,1),\;(1,4),\;(1,5),\;(2,1),\;(2,5),\;(3,1),\;(4,1),\;(4,2),\;(5,1),\;(5,2),\;(5,3),\;(5,4),\;(5,5).
$$
	Then we get that $\MNM(W^2)$ has the number $5$ (e.g.$R=\{1,4,5\}, C=\{1,5\}$).
	Hence the algorithm terminates immediately with $H=W^{2}$.
	The permutation $(1,4,2,5,3)$ selects the zeros
	$a_{1,4}=a_{2,5}=a_{3,1}=a_{4,2}=a_{5,3}=0$.
	Thus $\theta(H)=\frac{1}{4}(2+10+7+5)$, 
	hence the genus formula then gives
	$g(G)=\frac{1}{2}\theta(H)-\frac{1}{2}(5-1)=1$.
	This is equal to Seifert genus of the trefoil.
	 \hfill $\square$
\end{example}

\begin{example}[Conway knot]\label{ex:conway}
The $11\times 11$ grid diagram of Conway knot

 ($\sigma_\mathbb{O}=\{4,5,10,3,1,11,9,8,7,2,6\}$,$\sigma_\mathbb{X}=\{1,2,4,8,5,7,6,10,3,9,11\}$) gives the winding matrix

\addtocounter{MaxMatrixCols}{10}
$$
W(G)=
\begin{pmatrix}
0 & 0 & 0 & 0 & 0 & 0 & -1 & -1 & -1 & -1 & -1  \\
0 & 0 & 0 & -1 & -1 & -1 & -2 & -2 & -1 & -1 & -1  \\
0 & 0 & 0 & -1 & -1 & -1 & -2 & -3 & -2 & -2 & -1 \\
0 & 0 & 0 & -1 & 0 & 0 & -1 & -2 & -2 & -2 & -1 \\
0 & 0 & 0 & -1 & 0 & 0 & 0 & -1 & -1 & -2 & -1  \\
0 & 0 & 0 & -1 & 0 & 0 & 0 & 0 & 0 & -1 & 0  \\
0 & 0 & -1 & -2 & -1 & 0 & 0 & 0 & 0 & -1 & 0  \\
0 & -1 & -2 & -2 & -1 & 0 & 0 & 0 & 0 & -1 & 0  \\
0 & -1 & -2 & -2 & -2 & -1 & -1 & -1 & -1 & -1 & 0  \\
0 & -1 & -1 & -1 & -1 & 0 & 0 & 0 &  0 & 0 & 0  \\
0 & 0 & 0 & 0 & 0 & 0 & 0 & 0 & 0 & 0 & 0
\end{pmatrix}.
$$

\noindent\textbf{Phase~1:}	
First, get the non-negative matrix:
\addtocounter{MaxMatrixCols}{10}
$$
W^0=
\begin{pmatrix}
0 & 1 & 2 & 2 & 2 & 1 & 1 & 2 & 1 & 1 & 0  \\
0 & 1 & 2 & 1 & 1 & 0 & 0 & 1 & 1 & 1 & 0  \\
0 & 1 & 2 & 1 & 1 & 0 & 0 & 0 & 0 & 0 & 0  \\
0 & 1 & 2 & 1 & 2 & 1 & 1 & 1 & 0 & 0 & 0  \\
0 & 1 & 2 & 1 & 2 & 1 & 2 & 2 & 1 & 0 & 0  \\
0 & 1 & 2 & 1 & 2 & 1 & 2 & 3 & 2 & 1 & 1  \\
0 & 1 & 1 & 0 & 1 & 1 & 2 & 3 & 2 & 1 & 1  \\
0 & 0 & 0 & 0 & 1 & 1 & 2 & 3 & 2 & 1 & 1  \\
0 & 0 & 0 & 0 & 0 & 0 & 1 & 2 & 1 & 1 & 1  \\
0 & 0 & 1 & 1 & 1 & 1 & 2 & 3 & 2 & 2 & 1  \\
0 & 1 & 2 & 2 & 2 & 1 & 2 & 3 & 2 & 2 & 1
\end{pmatrix}.
$$

\noindent\textbf{Phase~2:}
All column minima and row minima are $0$, then  $W^2 = W^1 = W^0$.

\noindent\textbf{Phase~3:}

Zero positions:
\begin{equation*}
\begin{split}
& (1,1),\ (1,11),\ (2,1),\ (2,6),\ (2,7),\ (2,11),\ (3,1),\ (3,6),\ (3,7),\ (3,8),\ (3,9),\ (3,10),\\ 
& (3,11), \ (4,1), \ (4,9),\ (4,10),\ (4,11),\ (5,1),\ (5,10),\ (5,11),\ (6,1),\ (7,1),\ (7,4),\ (8,1),\\ 
& (8,2), \ (8,3),\ (8,4),\ (9,1),\ (9,2),\ (9,3),\ (9,4),\ (9,5),\ (9,6),\ (10,1),\ (10,2),\ (11,1)\ .
\end{split}
\end{equation*}
Then we get $\MNM(W^2)$ has the number $10$ (e.g. $R = \{2,3,4,5,7,8,9,10\}$, $C = \{1,11\}$).
Hence the algorithm continue, and we get $\varepsilon =1$. Thus it is obtained 
$$
W^3=
\begin{pmatrix}
0 & 0 & 1 & 1 & 1 & 0 & 0 & 1 & 0 & 0 & 0  \\
1 & 1 & 2 & 1 & 1 & 0 & 0 & 1 & 1 & 1 & 1  \\
1 & 1 & 2 & 1 & 1 & 0 & 0 & 0 & 0 & 0 & 1  \\
1 & 1 & 2 & 1 & 2 & 1 & 1 & 1 & 0 & 0 & 1  \\
1 & 1 & 2 & 1 & 2 & 1 & 2 & 2 & 1 & 0 & 1  \\
0 & 0 & 1 & 0 & 1 & 0 & 1 & 2 & 1 & 0 & 1  \\
1 & 1 & 1 & 0 & 1 & 1 & 2 & 3 & 2 & 1 & 2  \\
1 & 0 & 0 & 0 & 1 & 1 & 2 & 3 & 2 & 1 & 2  \\
1 & 0 & 0 & 0 & 0 & 0 & 1 & 2 & 1 & 1 & 2  \\
1 & 0 & 1 & 1 & 1 & 1 & 2 & 3 & 2 & 2 & 2  \\
0 & 0 & 1 & 1 & 1 & 0 & 1 & 2 & 1 & 1 & 1
\end{pmatrix}.
$$
Zero positions:
\begin{equation*}
\begin{split}
& (1,1),\ (1,2),\ (1,6),\ (1,7),\ (1,9),\ (1,10),\ (1,11),\ (2,6),\ (2,7),\ (3,6),\ (3,7),\ (3,8),\\
&(3,9),\ (3,10),\
(4,9),\ (4,10),\ (5,10),\ (6,1),\ (6,2),\ (6,4),\ (6,6),\ (6,10),\ (7,4),\ (8,2), \\
& (8,3),\ (8,4),\ (9,2),\ (9,3),\ (9,4),\ (9,5),\ (9,6),\ (10,2),\ (11,1),\ (11,2),\ (11,6)\ .
\end{split}
\end{equation*}
Then we get  $\MNM(W^2)$ has the number $11$ (e.g. $R = \{1,2,3,4,6,8,9,11\}$, $C = \{2,4,10\}$).
Hence the algorithm terminated with $H = W^3$.

The permutation $(1,11,6)(2,7,4,9,5,10)(3,8)$ selects the zeros $a_{1,11} = a_{2,7} = a_{3,8} = a_{4,9} = a_{5,10} = a_{6,1}= a_{7,4} =a_{8,3} = a_{9,5} = a_{10,2} = a_{11,6} = 0$. Thus $\theta(H) = \frac{1}{4}(5+12+7+5+4+9+15+7)$, then we get that 
$$
g(G) = \frac{1}{2}\theta(H) - \frac{1}{2}(11-1) = 3.
$$
 \hfill $\square$
\end{example}

In the first example the algorithm terminates after Phase~1. As shown in Example~\ref{ex:conway}, more complex diagrams require additional iterative updates, but the bounded potential argument and standard augmenting path techniques guarantee efficient eventual termination.

\section{Application 2: Grid moves and the search for minimal genus of knot}\label{sec:gridmoves}

Grid diagrams admit three types of elementary
moves---commutation, stabilization, and
destabilization---coined \emph{grid moves} collectively
(Definition~3.1.8 of~\cite{ozsvath2012}).  Cromwell's
theorem~\cite{cromwell2004} states that two planar grid diagrams
represent equivalent links if and only if one can be transformed
into the other by a finite sequence of grid
moves~\cite{ozsvath2012}.  Consequently, every grid diagram $G$ of
a knot $K$ can be reached from any other diagram of the same knot
by a suitable sequence of commutations, stabilizations and
destabilizations.

Unlike the Seifert genus $g(K)$, which is a knot invariant, the
quantity $g(G)$ computed from a single grid diagram $G$  depends on the particular diagram chosen.
Different grid diagrams representing the same knot may therefore
yield different values of $g(G)$.  The Seifert genus is recovered
as the minimum over all diagrams:
\[
g(K)=\min\{\,g(G)\mid G \text{ is a grid diagram for } K\,\}.
\]

Grid moves provide a mechanism for navigating the space of grid
diagrams.  In particular, destabilization can, in favourable
cases, reduce $g(G)$ and bring one closer to a genus-minimizing
diagram.  An analogous role is played by commutation, which
permutes the markings without changing the grid size and can
prepare a diagram for a subsequent destabilization.  Together,
they open the possibility of a descent algorithm that, starting
from a given grid diagram $G_0$ of a knot $K$, produces a
sequence
\[
G_0 \xrightarrow{\text{comm.}} G_1 \xrightarrow{\text{stab.}} G_2 \xrightarrow{\text{comm.}} \cdots \xrightarrow{\text{destab.}} G_m
\]
with maybe decreasing genus (or equal, or increasing) at each step, and
terminates at a diagram $G_m$ whose genus $g(G_m)$ equals the
Seifert genus $g(K)$.

Understanding how each type of grid move affects the winding
matrix $W(G)$, and therefore the minimal non-negative matrix $H$
and the value $\theta(H)$, is an essential step towards realising
such a descent algorithm.  The analysis of these effects is the
subject of ongoing work and will be presented in a subsequent
paper.  The present contribution provides the algebraic engine
(Algorithm~\ref{alg:main}) that computes $g(G)$ from
$W(G)$ in $\mathcal{O}(n^{3})$ time, which forms a necessary
building block for any move-based optimisation.


This raises a
natural question: given a grid diagram $G$, what features of its
winding matrix $W(G)$ determine whether a particular
stabilization or commutation will increase, decrease, or leave
unchangee genus $g(G)$?  We formulate this problem here and
indicate a promising direction of investigation suggested by the
algebraic machinery of Algorithm~\ref{alg:main}.

Let $G$ be an $n\times n$ grid diagram and $G'$ a stabilization of
$G$, so that $G'$ is an $(n+1)\times(n+1)$ grid diagram.  Write
$H=H(G)$ and $H'=H(G')$ for the minimal non-negative matrices
produced by Algorithm~\ref{alg:main}.  The genus formula
 gives
\[
g(G)=\frac12\,\theta(H)-\frac{n-1}{2},
\qquad
g(G')=\frac12\,\theta(H')-\frac{n}{2}.
\]
Their difference is therefore
\begin{equation}\label{eq:gengap}
	g(G')-g(G) = \frac12\bigl(\theta(H')-\theta(H)-1\bigr).
\end{equation}
Thus $g$ increases by $k$ if and only if
$\theta(H')-\theta(H)=2k+1$; it stays the same precisely when
$\theta(H')=\theta(H)+1$.

Because $\theta(H)$ is a sum of local averages of the entries of
$H$ around each $X$ and $O$ marking, the
change $\theta(H')-\theta(H)$ is controlled by how the new row and
column inserted during stabilization affect these local averages.
For a destabilization---the inverse operation---the sign is
reversed, and one obtains a descent move precisely when
$\theta(H)-\theta(H') > 1$.

We now consider the $O:NW$ case as an example. Let $G$ be a grid diagram with grid number $n$. 
After performing an $O:NW$ stabilization on $G$ at an $O$-marking $(n-i+1,\; j-1)$, denote by $G'$ the resulting
grid diagram, which has grid number $n+1$,as shown in Figure~\ref{fig:ONW}.

\begin{figure}[H]
	\centering
	
\begin{tikzpicture}[x=0.75pt,y=0.75pt,yscale=-1,xscale=1]

\draw    (76.43,321.32) -- (201.33,321.32) ;
\draw    (76.08,398.57) -- (162.72,398.57) -- (200.98,398.57) ;
\draw    (76.43,375.42) -- (90.74,375.42) -- (201.33,375.42) ;
\draw    (76.43,451.67) -- (201.33,451.67) ;
\draw    (201.33,321.32) -- (201.33,451.67) ;
\draw    (128.19,321.65) -- (128.19,383.53) -- (128.19,452) ;
\draw    (76.43,321.32) -- (76.43,451.67) ;
\draw    (151.55,320.32) -- (151.55,450.67) ;
\draw    (316.83,317.71) -- (473.61,317.71) ;
\draw    (316.89,419.99) -- (473.67,419.99) ;
\draw    (316.83,398.31) -- (473.61,398.31) ;
\draw    (316.83,465.86) -- (473.61,465.86) ;
\draw    (473.61,317.71) -- (473.61,465.86) ;
\draw    (366.32,318.05) -- (366.32,388.38) -- (366.32,466.19) ;
\draw    (316.83,317.71) -- (316.83,465.86) ;
\draw    (395.22,317.71) -- (395.22,465.86) ;
\draw    (421.78,318.19) -- (421.78,466.33) ;
\draw    (316.98,375.85) -- (358.21,375.85) -- (473.76,375.85) ;
\draw    (208.41,379.95) -- (225.33,380.23) ;
\draw [shift={(227.33,380.26)}, rotate = 180.93] [color={rgb, 255:red, 0; green, 0; blue, 0 }  ][line width=0.75]    (10.93,-3.29) .. controls (6.95,-1.4) and (3.31,-0.3) .. (0,0) .. controls (3.31,0.3) and (6.95,1.4) .. (10.93,3.29)   ;

\draw (131.53,378.6) node [anchor=north west][inner sep=0.75pt]   [align=left] {O};
\draw (401.14,378.59) node [anchor=north west][inner sep=0.75pt]   [align=left] {O};
\draw (85.5,380.39) node [anchor=north west][inner sep=0.75pt]   [align=left] {X};
\draw (132.79,426.66) node [anchor=north west][inner sep=0.75pt]   [align=left] {X};
\draw (373.66,401.07) node [anchor=north west][inner sep=0.75pt]   [align=left] {O};
\draw (326.37,379.04) node [anchor=north west][inner sep=0.75pt]   [align=left] {X};
\draw (375.35,438.1) node [anchor=north west][inner sep=0.75pt]   [align=left] {X};
\draw (400.59,400.35) node [anchor=north west][inner sep=0.75pt]   [align=left] {X};
\draw (125.79,297.11) node [anchor=north west][inner sep=0.75pt]    {$j-1$};
\draw (12.36,378.62) node [anchor=north west][inner sep=0.75pt]    {$n-i+1$};
\draw (249.96,400.04) node [anchor=north west][inner sep=0.75pt]    {$n-i+1$};
\draw (249.96,380) node [anchor=north west][inner sep=0.75pt]    {$n-i+2$};
\draw (359.68,292.98) node [anchor=north west][inner sep=0.75pt]    {$j-1$};
\draw (405.29,293.46) node [anchor=north west][inner sep=0.75pt]    {$j$};

\end{tikzpicture}

	\caption{$O:NW$}
	\label{fig:ONW}
\end{figure}

The following example is about the knot $5^1$. We can see that the genus increase $1$ after the stabilization O:NW which  occur at a given O-marking $(2,2)$.
\begin{example}
The $7\times 7$ grid diagram of knot $5^1$ ($\sigma_\mathbb{O}=\{1,2,3,4,5,6,7\}$, $\sigma_\mathbb{X}=\{5,6,1,2,7,3,4\}$) givens the following winding matrix
$$
W(G)=
\begin{pmatrix}
0 & 0 & 0 & 0 & 0 & 1 & 1   \\
0 & 0 & 1 & 1 & 1 & 2 & 1  \\
0 & 1 & 2 & 2 & 2 & 2 & 1 \\
0 & 1 & 2 & 2 & 1 & 1 & 0 \\
0 & 1 & 2 & 1 & 0 & 0 & 0  \\
0 & 1 & 1 & 0 & 0 & 0 & 0  \\
0 & 0 & 0 & 0 & 0 & 0 & 0 
\end{pmatrix}.
$$
We will see that  $W^2 = W^1 = W^0=W(G)$. Since $\MNM(W^2)=7$,
 $\theta(H) = \frac{1}{4}(5+10+21+7)$, and we get the associated genus
$$
g(G) = \frac{1}{2}\theta(H) - \frac{1}{2}(7-1) = 2.
$$

After the stabilization O:NW which  occur at a given O-marking $(2,2)$, we get a $8\times 8$ grid diagram ($\sigma_\mathbb{O}=\{1,2,3,4,5,6,7,8\}$, $\sigma_\mathbb{X}=\{6,7,2,1,3,8,4,5\}$), and its winding matrix is 
$$
W(G')=
\begin{pmatrix}
0 & 0 & 0 & 0 & 0 & 0 & 1 & 1   \\
0 & 0 & 1 & 1 & 1 & 1 & 2 & 1  \\
0 & 1 & 2 & 2 & 2 & 2 & 2 & 1 \\
0 & 1 & 2 & 2 & 2 & 1 & 1 & 0 \\
0 & 1 & 2 & 2 & 1 & 0 & 0 & 0  \\
0 & 1 & 2 & 1 & 0 & 0 & 0 & 0  \\
0 & 1 & 1 & 1 & 0 & 0 & 0 & 0  \\
0 & 0 & 0 & 0 & 0 & 0 & 0 & 0 
\end{pmatrix}.
$$
It is easy to see $\MNM(W^2)=8$ under the algorithm \ref{alg:main}, then $\theta(H') = \frac{1}{4}(2+10+28+7+5)$, and we get the associated genus
\[g(G') = \frac{1}{2}\theta(H') - \frac{1}{2}(8-1) = 3.\]
 \hfill $\square$
\end{example}
\subsection*{Acknowledgements}
	The author would like to thank Xiang Liu and Ziyi Liu at Hebei Normal University for discussions, and Xiaowei Pang for her encouragement. The project was funded by Science and Technology Project of Hebei Education Department  (Grant No.~QN2023030), Science Foundation of Hebei Normal University (Grant Nos.~L2022B02, L2025ZD04).
\subsection*{Data availability}
No data was used for the research described in the article.

\subsection*{Declaration}
\begin{itemize}
\item Python was used for the implementation of the algorithm.
\item AI was used for typesetting and grammar.
\end{itemize}

\bibliographystyle{alpha}
\bibliography{references}

\end{document}